\documentclass[12pt,a4paper]{article}
\usepackage{amsthm, amssymb, amsmath, latexsym}
\usepackage{verbatim}
\usepackage[all]{xy}
\usepackage{fouriernc}
\usepackage{color}
\usepackage{fullpage}
\usepackage{cite}
\newtheorem{defn}{Definition}
\newtheorem{lem}[defn]{Lemma}
\newtheorem{prop}[defn]{Proposition}

\newtheorem{thm}[defn]{Theorem}
\newtheorem{cor}[defn]{Corollary}

\newcommand{\lp}{\left(}
\newcommand{\rp}{\right)}

\DeclareMathOperator{\Real}{Re}
\DeclareMathOperator{\Imaginary}{Im}

\usepackage[pdftex]{hyperref}

\author{Dongho Byeon\footnote{The first author was supported
by Basic Science Research
Program through the National Research Foundation of Korea (NRF)
funded by the Ministry of Education (NRF-2013R1A1A2007694).} and
Keunyoung Jeong}

\title{The exceptional set in a generalized Goldbach's problem}

\date{}

\begin{document}
\maketitle \abstract{In this paper, we compute the size of the
exceptional set in a generalized Goldbach problem and show that for
a given polynomial $f\lp x \rp \in \mathbb{Z}[x]$ with a positive
leading coefficient, positive integers $A$, $B$, $g$, $i$, $j$ with
$0 < i, j < g$, there are infinitely many positive integers $n$
which satisfy $2f(n) = Ap_1 + Bp_2$ for primes $p_1 \equiv i,
p_2\equiv j \pmod{g}$ under a mild condition.}

\section{Introduction}
Br\"udern, Kawada and Wooley \cite{BKW} computed the size of the
exceptional set of a polynomial-type generalization of Goldbach
problem.

\begin{thm}
{\rm \cite[Theorem 1]{BKW}} Let $f\lp x \rp \in \mathbb{Z}[x]$ be a
polynomial which has a positive leading coefficient with degree $k$
and $\mathcal{E}_k(N, f)$ be the number of positive integers $n$
with $1 \leq n \leq N$ for which the equation $2f(n) = p_1 + p_2$
has no solution in primes $p_1$, $p_2$. Then there is an absolute
constant $c>0$ such that  $$\mathcal{E}_k(N,f) \ll_f
N^{1-\frac{c}{k}}.$$
\end{thm}

\noindent This theorem implies that there are infinitely many
positive integers $n$ which satisfy $2f(n) = p_1 + p_2$ for primes
$p_1$, $p_2$. Similarly, one can ask if for given positive integers
$A$, $B$, $g$, $i$, $j$ with $0 < i, j < g$, there are infinitely
many positive integers $n$ which satisfy $2f(n) = Ap_1 + Bp_2$ for
primes $p_1 \equiv i, \,p_2 \equiv j \pmod{g}$.

To answer this question, we will prove the following theorem.

\begin{thm} \label{thm2}
Let $f(x) \in \mathbb{Z}[x]$ be a polynomial which has a positive
leading coefficient with degree $k$. Let $A, B$ be positive odd
integers and $g$, $i$, $j$ positive integers with $0 < i, j <
g<N^{24\delta k}$ for a sufficiently small positive real number
$\delta$ to be chosen later and $(i,g)=(j,g)=1$. Suppose that there
is at least one integer $m$ such that
$$
2f(m) \equiv Ai + Bj \pmod{g}.
$$
Let $\Gamma=\{A,B,g,i,j\}$ and let $\mathcal{E}_{k,\Gamma}(N,f)$ be
the number of positive integers $n \in [1, N]$ with $2f(n) \equiv Ai
+ Bj \pmod{g}$ for which the equation $2f(n) = Ap_1 + Bp_2$ has no
solution in primes $p_1 \equiv i, p_2 \equiv j \pmod{g}$. Then there
is an absolute constant $c>0$ such that
$$
\mathcal{E}_{k,\Gamma}(N,f) \ll_{k,\Gamma} N^{1-\frac{c}{k}}.
$$
\end{thm}

This immediately implies the positive answer of the above question.

\begin{cor} \label{cor3}
Let $f(x) \in \mathbb{Z}(x)$ be a polynomial which has a positive
leading coefficient. Let $A, B$ be positive odd integers and $g$,
$i$, $j$ be positive integers with $0 < i, j < g$ and
$(i,g)=(j,g)=1$. If there is at least one integer $m$ such that
$$
2f(m) \equiv Ai + Bj \pmod{g},
$$
then there are infinitely many positive integers $n$ which satisfy
$$
2f(n) = Ap_1 + Bp_2
$$
for primes $p_1 \equiv i, \, p_2 \equiv j \pmod{g}$.
\end{cor}

Let $N$ be a large positive integer, $\delta$ a sufficiently small
positive real number to be chosen later, $X:=2f(N)$, $P:=
X^{6\delta}$, $Q: = X/P$ and $\kappa:=2^{-\frac{1}{k}}$. Let $A, B$
be positive odd integers and $g$, $i$, $j$ positive integers with $0
< i, j < g < P^4$ and $(i,g)=(j,g)=1$. Let $\Gamma=\{A,B,g,i,j\}$.
We define the exponential sum $S_i(\alpha)$ by
$$
S_i(\alpha) := \sum_{\substack{P < p \leq X \\ p\equiv i \pmod{g}}}
(\log{p})e(\alpha p),
$$
where $e(\alpha
p):=e^{2\pi\alpha p i}$ and the summation is over primes $p$ with $P
< p \leq X$ and $p\equiv i \pmod{g}$. When $T\subseteq [0,1]$, we
write
$$
r_{\Gamma}(n;T) :=
\int_{T} S_{i}\lp A\alpha \rp S_{j}(B\alpha)e(-\alpha n)d\alpha
$$
and $r_{\Gamma}(n) := r_{\Gamma}(n;[0,1])$. Then $r_{\Gamma}(2f(n))$
counts the number of solutions of the equation $2f(n) = Ap_1+Bp_2$
in primes $p_1 \equiv i, \, p_2 \equiv j \pmod{g}$ with weight
$\log{p_1}\log{p_2}$.

Let $\mathfrak{M} \subset [0,1]$ be the major arc defined by
$$\mathfrak{M} = \bigcup_{\substack{0\leq a \leq q \leq P \\ (a,q) = 1}}
\mathfrak{M}(q,a),$$ where
$$ \mathfrak{M}(q,a) = \left\{\alpha \in [0,1]:
\left|\alpha- \frac{a}{q}\right|
\leq \frac{P}{qX} \right\},$$
and $\mathfrak{m} \subset [0,1]$ be the minor arc defined by
$$\mathfrak{m} = [0,1] \setminus\mathfrak{M}.$$

In Section 2 we compute $r_{\Gamma}(2f(n);\mathfrak{m})$,  in
Section 3 we compute $r_{\Gamma}(2f(n);\mathfrak{M})$, and in
Section 4 combining these, we prove Theorem \ref{thm2}. Basically we follow
\cite{BKW} and \cite{MV}.

Finally we mention that some special forms of Corollary \ref{cor3} are
applied to the arithmetic of elliptic curves. See \cite{BJ} and \cite{BJK}.
One of the aims of this paper is to give a full proof of a full
generalization of the special forms for future applications.

\section{Minor Arc}

In \cite[Lemma 1]{BKW}, the authors proved that there exists a
positive real number $a=a(\delta)$ depending on $\delta$ such that
$$ \sum_{\kappa N < n \leq N}
\left| r(2f(n);\mathfrak{m})\right| \ll XN^{1-\frac{a}{k}},$$ where
$r(2f(n);\mathfrak{m})= \int_{\mathfrak{m}}S\lp\alpha\rp^2 e(-\alpha
\cdot 2f(n))d\alpha$ and $S(\alpha) = \sum_{\substack{P<p \leq X }}
(\log{p})e(\alpha p)$.

In this section, we show that the same result holds for
$r_{\Gamma}(2f(n);\mathfrak{m})$. To do this, we need the following
lemma which concerns the residue class condition; $i, j \pmod{g}$
and the coefficient condition; $A$, $B$. For the proof of the lemma,
we follow the proof of \cite[Theorem 13.6]{IK}. A new ingredient in
our proof is the orthogonality relations of Dirichlet characters.

\begin{lem}\label{Minorlem1}
Suppose that there exist integers $a$ and $q$ such that $(a,q) = 1$
and $\left|\alpha - \frac{a}{q}\right| \leq \frac{1}{q^2}$. Then for
$x \geq 2$ we have
$$
\sum_{\substack {p\leq x \\ p \equiv i \pmod{g}}} (\log{p})e(\alpha
Ap) \ll (x^{\frac{4}{5}}+ xq^{-\frac{1}{2}} +
x^{\frac{1}{2}}q^{\frac{1}{2}})(\log{x})^3,
$$
where the summation is over primes $p$ with $p \leq x$ and $p\equiv
i \pmod{g}$.
\end{lem}

\begin{proof}
Let $\chi$ be a Dirichlet character modulo $g$. The orthogonality
relations of Dirichlet characters imply that
\begin{eqnarray*}
\sum_{\substack {p \leq x \\
p \equiv i \pmod{g}}}(\log{p})e(\alpha Ap) &= &\sum_{p \leq
x}\frac{1}{\varphi(g)}\sum_{\chi}\bar{\chi}(i)\chi(p)
(\log{p})e(\alpha Ap) \\ &\ll&  \sum_{\chi}\left| \sum_{p \leq
x}\chi(p) (\log{p})e(\alpha Ap)\right|.
\end{eqnarray*}

\noindent Thus it is enough to show that
$$
\sum_{p \leq x}\chi(p) (\log{p})e(\alpha Ap) \ll (x^{\frac{4}{5}} +
xq^{-\frac{1}{2}} + x^{\frac{1}{2}}q^{\frac{1}{2}})(\log{x})^3.
$$

\noindent Let $\Lambda(n)$ be the von Mangoldt function which
defined as follows:
$$
\Lambda(n) = \left\{ \begin{array}{ccc}
\log{p} & \textrm{if $n = p^k$,} \\
0 & \textrm{otherwise.}
\end{array}\right.
$$

\noindent By the fact $\displaystyle\sum_{p \leq x}\log{p} \ll x$,
$$
\bigg| \sum_{\substack{p^{k} \leq x }}\chi(p^k)(\log{p})e(\alpha
Ap^{k})\bigg| \leq \sum_{p^{k} \leq x} \log{p} = \sum_{p \leq
x^{\frac{1}{k}}} \log{p} \ll x^{\frac{1}{k}}.
$$

\noindent Hence
$$
\sum_{p \leq x}\chi(p) (\log{p})e(\alpha Ap) = \sum_{n \leq x}
\chi(n)\Lambda(n)e(\alpha An) + O(x^{\frac{1}{2}}).
$$

\noindent Thus it is enough to show that
$$
\sum_{n \leq x}\chi(n)\Lambda(n)e(\alpha An)
= \sum_{\substack{n \leq Ax \\
A|n}}\chi(\frac{n}{A})\Lambda(\frac{n}{A})e(\alpha n) \ll
(x^{\frac{4}{5}} + x q^{-\frac{1}{2}} +
x^{\frac{1}{2}}q^{\frac{1}{2}})(\log{x})^3.
$$

\noindent From the Vaughan's identity, we have that for $y, z \geq
1$ and $n$ such that $A|n$ and $\frac{n}{A} > z$,
$$
\Lambda(\frac{n}{A}) = \sum_{\substack{b|\frac{n}{A} \\
b\leq y }}\mu(b)\log{\frac{n}{Ab}}
- \mathop{\sum\sum}_{\substack{bc | \frac{n}{A} \\ b \leq y, c \leq
z}}\mu(b)\Lambda(c) + \mathop{\sum\sum}_{\substack{bc | \frac{n}{A}
\\ b > y, c > z}}\mu(b)\Lambda(c).
$$

\noindent Then
\begin{eqnarray}
\sum_{\substack{n \leq Ax \\
A|n}}\chi(\frac{n}{A})\Lambda(\frac{n}{A})e(\alpha n) &=&
\mathop{\sum\sum}_{\substack{lm \leq Ax, A|l \\ m \leq M}}
\chi(\frac{l}{A}m)\mu(m)(\log{\frac{l}{A}})e(\alpha lm) \nonumber \\
&-& \mathop{\sum\sum\sum}_{\substack{lmn \leq Ax, A|l\\ m \leq M, n
\leq N}} \chi(\frac{l}{A}mn)\mu(m)\Lambda(n)e(\alpha lmn) \nonumber
\\ &+& \mathop{\sum\sum\sum}_{\substack{lmn \leq Ax, A|l\\ m \geq M, n
\geq N}} \chi(\frac{l}{A}mn)\mu(m)\Lambda(n)e(\alpha lmn) + O(N).
\nonumber
\end{eqnarray}

\noindent We need
\begin{equation}
\sum_{m \leq M} \left|\sum_{\substack{mn \leq x \\
A|n}}\chi(\frac{n}{A})e(\alpha mn)\right| \ll (M + xq^{-1} +
q)\log{2qx}. \label{useful}
\end{equation}

\noindent It is derived as follow:
\begin{eqnarray}
&&\sum_{ m \leq M}\left|\sum_{\substack{mn \leq x \\
A|n}}\chi(\frac{n}{A})e(\alpha mn)\right| =  \sum_{m \leq
M}\left|\sum_{i=0}^{g-1}\chi(i)\sum_{\substack{mAn' \leq x \\
n'\equiv i \pmod{g}}}
e(\alpha mn')\right| \nonumber \\
& & \leq g \sum_{m \leq M}\min(\frac{x}{mA}, \frac{1}{2\|mA\alpha
\|})  \leq g \sum_{m \leq AM}\min(\frac{x}{m}, \frac{1}{2\|m\alpha
\|})
\nonumber \\
& & \ll (AM + xq^{-1} + q)\cdot\log{2qAx} \ll (M + xq^{-1} +
q)\cdot\log{2qx}, \nonumber
\end{eqnarray}

\noindent where $\| \alpha \| = \displaystyle
\min_{u\in\mathbb{Z}}|\alpha-u|$. It is known \cite[Theorem
13.6]{IK} that
$$
\sum_{n \leq x} \Lambda(n)e(\alpha n) \ll (x^{\frac{4}{5}} +
xq^{-\frac{1}{2}} + x^{\frac{1}{2}}q^{\frac{1}{2}})(\log{x})^3.
$$

\noindent If we use (\ref{useful}) instead of (13.46) in the proof of
\cite[Theorem 13.6]{IK} and take $M=N=x^{\frac{2}{5}}$, then by the
same argument in the proof of \cite[Theorem 13.6]{IK}, we get
\begin{eqnarray*}
&&\mathop{\sum\sum}_{\substack{lm \leq Ax, A|l \\ m \leq M}}
\chi(\frac{l}{A}m)\mu(m)(\log{\frac{l}{A}})e(\alpha lm) \ll
(x^{\frac{2}{5}} + xq^{-1} + q)\log{qx}\cdot \log{x},\\
&&\mathop{\sum\sum\sum}_{\substack{lmn \leq Ax, A|l\\ m \leq M, n
\leq N}} \chi(\frac{l}{A}mn)\mu(m)\Lambda(n)e(\alpha lmn) \ll
(x^{\frac{4}{5}} + xq^{-1} + q)\log{qx}\cdot \log{x},\\
&&\mathop{\sum\sum\sum}_{\substack{lmn \leq Ax, A|l\\ m \geq M, n
\geq N}}\chi(\frac{l}{A}mn)\mu(m)\Lambda(n)e(\alpha lmn)\ll
(x^{\frac{4}{5}} + xq^{-\frac{1}{2}} +
x^{\frac{1}{2}}q^{\frac{1}{2}})(\log{x})^3,
\end{eqnarray*}
which prove the lemma.
\end{proof}

Now we can prove the following analogue of \cite[Lemma 1]{BKW}. For
the proof, we follow the proof of \cite[Lemma 1]{BKW}. A new
ingredient in our proof is the bound of $S_i(A\alpha)$ in Lemma
\ref{Minorlem1}.

\begin{prop} \label{Minormain}
There is a positive real number $ a =a(\delta)$ such that
$$
\sum_{\kappa N < n \leq N} |r_{\Gamma}(2f(n);\mathfrak{m})| \ll
XN^{1-\frac{a}{k}}.
$$
\end{prop}

\begin{proof}
By the H\"older inequality, we have
\begin{eqnarray}
&&\sum_{\kappa N < n \leq N}
\left|r_{\Gamma}(2f(n);\mathfrak{m})\right| \nonumber\\
&\leq& \sup_{\alpha \in
\mathfrak{m}}|S_{i}(A\alpha)S_{j}(B\alpha)|^{\frac{1}{t}} \big(
\int_{0}^{1}|S_{i}(A\alpha)S_{j}(B\alpha)|d\alpha
\big)^{1-\frac{1}{t}} \big( \int_{0}^{1}|K(-\alpha)|^{t}d\alpha
\big)^{\frac{1}{t}}, \nonumber
\end{eqnarray}

\noindent where
$$
 K(\alpha) = \sum_{\kappa N < n \leq N}
\eta(2f(n))e(2f(n)\alpha) \qquad \textrm{and} \qquad \eta(u) =
\left\{ \begin{array}{lll}  1 & \textrm{ if }
r_{\Gamma}(u,\mathfrak{m}) \geq 0,
\\ -1 & \textrm{ otherwise. }\end{array} \right.
$$

\noindent By Lemma \ref{Minorlem1}, we get for $\alpha \in \mathfrak{m}$,
$$
S_{i}(A\alpha)=\sum_{\substack{P < p \leq X \\ p \equiv i \pmod{g}}}
(\log{p})e(\alpha Ap) \ll (X^{\frac{4}{5}} + Xq^{-\frac{1}{2}} +
X^{\frac{1}{2}}q^{\frac{1}{2}}) (\log{X})^3 \ll
X^{1-3\delta}(\log{X})^3.
$$

\noindent This implies that
$$\sup_{\alpha \in \mathfrak{m}}
\left|S_{i}(A\alpha)S_{j}(B\alpha)\right| \ll [X^{1-3\delta}(\log{X})^3]^2 .
$$

\noindent From the proof of \cite[Lemma 1]{BKW}, we know that
$$
\int_{0}^{1}\left|S_{i}(A\alpha)S_{j}(B\alpha)\right| d\alpha \ll
X\log{X} \qquad \textrm{and} \qquad
\int^{1}_{0}|K(-\alpha)|^{t}d\alpha \ll N^{t - k(1-\delta)}.
$$

\noindent By combining these bounds, we have
\begin{eqnarray*}
\sum_{\kappa N < n \leq N} |r_{\Gamma}(2f(n),\mathfrak{m})| &\ll&
(X\log{X})^{1-\frac{1}{t}}N^{1-\frac{(1-\delta)k}{t}}
(X^{1-3\delta}(\log{X})^3)^{\frac{2}{t}} \\ &\ll&
NX^{1-\frac{5\delta}{t}}(\log{X})^{2},
\end{eqnarray*}
which proves the lemma.
\end{proof}

\section{Major Arc}

Let $Y$ be a real number with $1 \leq Y \leq X^{\frac{\delta} {k}}$.
In \cite[Lemma 2]{BKW}, the authors proved that for all $n$
satisfying $\kappa N < n \leq N$, with the possible exception of
$O(N^{1+\epsilon}Y^{-1})$ values of $n$
$$
r(2f(n);\mathfrak{M}) \gg XY^{-\frac{1}{2}}(\log X)^{-1},
$$
where $r(2f(n);\mathfrak{M})= \int_{\mathfrak{M}}S\lp\alpha\rp^2
e(-\alpha \cdot 2f(n))d\alpha$ and $S(\alpha) = \sum_{\substack{P<p
\leq X }} (\log{p})e(\alpha p)$.

In this section, we show that the same result holds for
$r_{\Gamma}(2f(n);\mathfrak{M})$. To do this, we need some lemmas
which concern the residue class condition; $i, j \pmod{g}$ and the
coefficient condition; $A$, $B$.

First we state the basic properties of exceptional characters, which
are established by Davenport \cite{Dav}.

\begin{lem}
There is a constant $c_1 > 0$ such that $L(\sigma, \chi) \neq 0$
whenever
$$
\sigma \geq 1-\frac{c_1}{\log{P}},
$$
for all primitive Dirichlet characters $\chi$ of modulus $q \leq P$,
with the possible exception of at most one primitive character
$\tilde{\chi}$ $\pmod{ \tilde{r}}$. If it exists, $\tilde{\chi}$ is
quadratic, and the unique exceptional real zero $\tilde{\beta}$ of
$L(s, \tilde{\chi})$ satisfies
\begin{equation*}
\frac{c_2}{\tilde{r}^{\frac{1}{2}}\log^2{\tilde{r}}} \leq 1-\tilde{\beta}
\leq \frac{c_1}{\log{P}}, \label{Lzero}
\end{equation*}
for a constant $c_2>0$.
\end{lem}

The following lemma is a modification of \cite[Theorem 7]{Ga}. For
the proof of the lemma, we follow the proof of \cite[Theorem 7]{Ga}.
A new ingredient in our proof is the orthogonality relations of
Dirichlet characters.

\begin{lem}\label{Majorlem1}
Suppose that $\frac{x}{P} \leq h \leq x$ and
$\exp(\log^{\frac{1}{2}}{x}) \leq P \leq x^b$. If there is no
exceptional character, we have
$$
\sum_{q \leq P}\sum_{\chi}\!{}^* \sum_{\substack{x \\ p \equiv i
\pmod{g}}}^{x+h} \chi(p)\log{p} \ll
h\exp(-c_3\frac{\log{x}}{\log{P}})
$$
for a constant $c_3$, where $\sum\!{}^*$ denotes that the sum is
taken over all primitive Dirichlet characters of modulus $q$ and if
there is the exceptional character, the right hand side may be
replaced by $ h(1-\tilde{\beta})\log{P} \exp(-c_3
\frac{\log{X}}{\log{P}}). $ Here the term with $q=1$ is read as
follows: if there is no exceptional character, it is
$$
\sum_{\substack{x \\ p \equiv i \pmod{g}}}^{x+h}\log{p} -
\sum_{\substack{x < n \leq x+h  \\ n \equiv i \pmod{g}}}1
$$
and if there is the exceptional character, it is
$$
\sum_{\substack{x \\ p \equiv i
\pmod{g}}}^{x+h}\tilde{\chi}(p)\log{p} + \sum_{\substack{x < n \leq
x+h \\ n \equiv i \pmod{g}}}n^{\tilde{\beta}-1}.
$$
\end{lem}

\begin{proof}
Let
$$
\psi(x) := \sum_{n\leq x}\Lambda(n), \,\,\, \psi(x,\chi) := \sum_{n
\leq x}\chi(n)\Lambda(n), \,\,\, \mbox{and} \,\,\, \psi(x, \chi; i,
g) := \sum_{\substack{n \leq x \\ n \equiv i
\pmod{g}}}\chi(n)\Lambda(n).
$$

\noindent Using the orthogonality relations of Dirichlet characters,
we have
\begin{equation}
\psi(x, \chi; i, g) = \frac{1}{\varphi(g)} \sum_{n \leq
x}(\sum_{\chi^{\prime} }
\bar{\chi^{\prime}}(i)\chi^{\prime}(n))\chi(n)\Lambda(n) =
\frac{1}{\varphi(g)} \sum_{\chi^{\prime}} \bar{\chi^{\prime}}(i)
\psi(x, \chi\cdot\chi^{\prime}), \label{Psi}
\end{equation}

\noindent where $\chi^{\prime}$ varies in the set of Dirichlet
characters of modulus $g$ and $\chi\cdot\chi'(n) := \chi(n)\chi'(n)$.
For $q \leq T \leq x^{\frac{1}{2}}$
$$
\psi(x,\chi) = \delta_{\chi}x - \sum_\rho \frac{x^{\rho}}{\rho} +
O(\frac{x\log^2{x}}{T}),
$$

\noindent where $\delta_{\chi}=1$ or 0 according to whether $\chi =
\chi_{0}$ or not, and the sum on the right is over the zeros $\rho$
of $L(s,\chi)$ in $0 \leq \Real(\rho) \leq 1$,
$|\Imaginary(\rho)|\leq T$. By (\ref{Psi})
\begin{equation}
\psi(x, \chi; i, g) = \frac{1}{\varphi(g)} \sum_{\chi^{\prime} }
\bar{\chi^{\prime}}(i) (\delta_{\chi\cdot\chi^{\prime}}x -
\sum_{\rho}\frac{x^{\rho}}{\rho})
 + O(\frac{x\log^2{x}}{T}), \label{majoreq1}
\end{equation}

\noindent where the second sum is over the zeros $\rho$ of
$L(s,\chi\cdot\chi^{\prime})$ in $0 \leq \Real(\rho) \leq 1$,
$|\Imaginary(\rho)|\leq T$.
Since
\begin{equation*}
\label{majoreq2} \psi(x+h,\chi;i,g) - \psi(x,\chi;i,g) =
\sum_{\substack{x
\\ p\equiv i \pmod{g}}}^{x+h}\chi(p)\log{p} + O(x^{\frac{1}{2}}),
\end{equation*}

\noindent by (\ref{majoreq1})
\begin{equation}
\sum_{\substack{x \\ p\equiv i \pmod{g}}}^{x+h}\chi(p)\log{p} =
\frac{1}{\varphi(g)}(\sum_{\chi^{\prime}}\bar{\chi^{\prime}}
(i)(\delta_{\chi \cdot \chi^{\prime}}h -
\sum_{\rho}\frac{(x+h)^{\rho}-x^{\rho}}{\rho})) + O(\frac{x\log^2
x}{T})\nonumber.
\end{equation}

\noindent Thus
\begin{equation}
\sum_{q \leq P}\sum_{\chi}\!{}^{*}\sum_{\substack{x
\\ p\equiv i \pmod{g}}}^{x+h}\chi(p)\log{p}
\ll \sum_{q \leq P}\sum_{\chi}\!{}^{*}\sum_{\chi^\prime } h(\sum_{\rho}
x^{\beta-1} + \frac{P^2}{T}), \label{majoreq4}
\end{equation}

\noindent where the fourth sum of the right hand side is over the
zeros $\rho = \beta + \gamma i$ of $L(s,\chi\cdot\chi^{\prime})$ in
$0 \leq \Real(\rho) \leq 1$, $|\Imaginary(\rho)|\leq T$. Let
$N_{\chi}(\alpha,T)$ be the number of zeros $\rho$ of $L(s,\chi)$ in
the rectangle
$$
\{ \rho \in \mathbb{C} : \alpha \leq \textrm{Re($\rho$)} \leq 1,
|\textrm{Im($\rho$)}| \leq T \}.
$$

\noindent Then the quadruple sum on the right hand side of
(\ref{majoreq4}) is
\begin{eqnarray}
&-&\int^{1}_{0} x^{\alpha - 1} \frac{\partial}{\partial\alpha}
(\sum_{q \leq P}\sum_{\chi}\!{}^{*}\sum_{\chi^\prime}
N_{\chi\cdot\chi^\prime}(\alpha, T))d\alpha \nonumber
\\
&=& \int^{1}_{0}x^{\alpha-1}\log{x}\sum_{q \leq
P}\sum_{\chi}\!{}^{*}\sum_{\chi^\prime} N_{\chi\cdot\chi^\prime}(\alpha,T)d\alpha
+ \frac{1}{x}\sum_{q \leq P}\sum_{\chi}\!{}^{*}\sum_{\chi^\prime}
N_{\chi\cdot\chi^\prime}(0,T).
\label{Nval}
\end{eqnarray}

\noindent Consider the decomposition of a character $\chi$ of
modulus $q = \prod_{p}p^{{a_p}}$
$$\chi = \prod_{p}\chi_{p^{a_p}},$$
where $\chi_{p^{a_p}}$ is a character of modulus $p^{a_p}$. Assume $g$ is a
prime and $g|q$. Then the conductor of $\chi\cdot\chi'$ is $q$
except the case that $v_{g}(q)=1$ and the $p$-parts of
decompositions of $\chi$ and $\chi'$ are inverse each other, in this
case the conductor of $\chi\cdot\chi'$ is $\frac{q}{g}$. Therefore,
$$
\sum_{\chi}\!{}^{*}\sum_{\chi'}N_{\chi\cdot\chi'}(a,T)
 = (\varphi(g)-1)(\sum_{\chi_1}N_{\chi_1}(a,T)
 + \sum_{\chi_2}\!{}^{*}N_{\chi_2}(a,T)),
$$
where the first sum on the right hand side varies in the set of
non-primitive characters of modulus $q$ which are induced by a
primitive character of modulus $\frac{q}{g}$ and the second one on
the right hand side varies in the set of primitive characters of
modulus $q$. Let $\chi$ be a Dirichlet character of modulus $q$
induced by primitive character $\chi^{*}$ of modulus $q^{*}$. Then,
$$
L(s,\chi)\prod_{\substack{p|q \\ p \nmid q^{*}}}
\frac{1}{1-\chi^{*}(p)p^{-s}} = L(s,\chi^{*}).
$$
Each factor $\frac{1}{1-\chi^{*}(p)p^{-s}}$ has a pole at
$$
s=\frac{2\pi i }{\log{p}}(l + \frac{m_p}{\varphi(q^*)}),
$$
where $m_{p}$ be the smallest positive integer satisfying
$\chi^{*}(p) = e(\frac{m_p}{\varphi(q^*)})$ and $l$ be an
integer such that $|\frac{2\pi}{\log{p}}(l + \frac{m_p}{\varphi(q^*)})| < T$.
Therefore
$$
N_{\chi}(a,T) = N_{\chi^{*}}(a,T) + c_{a,q,q^*,T},
$$
where $c_{a,q,q^*,T}=0$ if $a>0$ and
$c_{0,q,q^*,T}=\sum_{\substack{p|q \\ p\nmid q^*}} ([\frac{T \log{p}
}{\pi}]+1)$. Hence, for a prime $g$
\begin{eqnarray}
\sum_{\substack{q \leq P \\ g|q }} \sum_{\chi}\!{}^{*}\sum_{\chi'}
N_{\chi\cdot\chi'}(a,T) &=& (\varphi(g)-1)(\sum_{q \leq
\frac{P}{g}}\sum_{\chi}\!{}^{*} (N_{\chi}(a,T) +
c_{a,q,\frac{q}{g},T})
+ \sum_{q\leq P}\sum_{\chi}\!{}^{*}N_{\chi}(a,T)) \nonumber \\
& < & 2\varphi(g)\sum_{q \leq P}\sum_{\chi}\!{}^{*}N_{\chi}(a,T)
+ \varphi(g)\sum_{q \leq \frac{P}{g}}\sum_{\chi}\!{}^{*}c_{a,q,\frac{q}{g},T}.
\nonumber
\end{eqnarray}

\noindent Similarly one can prove its generalization to a composite
$g$
$$ \sum_{\substack{q \leq P \\ (q,g)=m}}\sum_{\chi}\!{}^{*}\sum_{\chi'}
N_{\chi\cdot\chi'}(a,T) <
d(m)\varphi(g)\sum_{q \leq gP}\sum_{\chi}\!{}^{*}N_{\chi}(a,T)
+d(m)\varphi(g)\sum_{q \leq gP}\sum_{\chi}\!{}^{*}c_{a,q,\frac{q}{m},T},$$

\noindent where $d(m)$ be the number of divisors of $m$. Therefore
if $a>0$,
\begin{eqnarray}
\sum_{q \leq P}\sum_{\chi}\!{}^{*}\sum_{\chi'}N_{\chi\cdot\chi'}(a,T)
& = & \sum_{\substack{q \leq P\\(q,g)=1}}
\sum_{\chi}\!{}^{*}\sum_{\chi'}N_{\chi\cdot\chi'}(a,T)
+\sum_{\substack{q \leq P\\(q,g)>1}}\sum_{\chi}\!{}^{*}\sum_{\chi'}N_{\chi\cdot\chi'}(a,T)
\nonumber \\
& \leq & \sum_{q \leq gP}\sum_{\chi}\!{}^{*}N_{\chi}(a,T)
+ d(g)^2\varphi(g)\sum_{q \leq gP}\sum_{\chi}\!{}^{*}N_{\chi}(a,T)
\nonumber \\
& \ll & \sum_{q \leq gP}\sum_{\chi}\!{}^{*}N_{\chi}(a,T),\nonumber
\end{eqnarray}
and if $a=0$ since
$c_{0,q,\frac{q}{m},T} \leq \frac{T\log{m}}{\pi} +d(m)$,
$$
\sum_{q \leq
P}\sum_{\chi}\!{}^{*}\sum_{\chi'}N_{\chi\cdot\chi'}(0,T) \ll \sum_{q
\leq gP}\sum_{\chi}\!{}^{*}N_{\chi}(0,T)+\sum_{q \leq
gP}\sum_{\chi}\!{}^{*}(\frac{T\log{g}}{\pi} +d(g)).
$$
Thus (\ref{Nval}) is
\begin{eqnarray}
&\ll& \int^{1}_{0}x^{\alpha -1}\log{x}\sum_{q \leq gP}
\sum_{\chi}\!{}^{*}N_{\chi}(\alpha, T)d\alpha +\frac{1}{x}(\sum_{q\leq
gP}\sum_{\chi}\!{}^{*}N_{\chi}(0, T) +
\sum_{q \leq gP}\sum_{\chi}\!{}^{*}T) \nonumber\\
&\leq& \int^{1-\theta(T)}_{0}x^{\alpha -1}\log{x}\sum_{q \leq gP}
\sum_{\chi}\!{}^{*}N_{\chi}(\alpha, T)d\alpha +\frac{1}{x}\sum_{q\leq
gP}\sum_{\chi}\!{}^{*}N_{\chi}(0, T) + \frac{gP^2T}{x}
\nonumber\\
&\ll& x^{-\frac{1}{2} \theta(T)} + \frac{P^2T}{x}, \label{theta}
\end{eqnarray}

\noindent where
$$
\theta(T) = \left\{ \begin{array}{lll}  \displaystyle\frac{1}{\log{T}} & \textrm{
if there is no exceptional character,}
\\ \displaystyle\frac{c_2}{\log{T}}\log{\frac{e\cdot
c_1}{(1-\tilde{\beta})\log{T}}}& \textrm{ otherwise. }\end{array}
\right.
$$

\noindent For (\ref{theta}), we used \cite[Theorem 6]{Ga};
$$
\sum_{q \leq T}\sum_{\chi}\!{}^{*} N_{\chi}(\alpha, T) \ll
T^{c(1-\alpha)}
$$

\noindent and assumed $T^{c} \leq x^{\frac{1}{2}}$ and $T>gP$. If we
choose $T = P^5$ and $b = \frac{1}{10c}$,
then $\frac{P^2T}{x} \ll x^{-\frac{1}{2}}$ and the lemma follows.
\end{proof}

From Lemma \ref{Majorlem1} and the argument below \cite[Lemma 4.3]{MV}, we have
the following modification of \cite[Lemma 4.3]{MV}.

\begin{lem} \label{Majorlem2}
If the exceptional character does not occur, there are positive
absolute constants $c_4, c_5$ which satisfy
$$\sum_{q \leq P}\sum_{\chi}\!{}^* \max_{x \leq N}\max_{h \leq N}
(h + \frac{N}{P})^{-1} \bigg|\sum_{\substack{x-h \\ p \equiv i
\pmod{g}}}^{x} \chi(p)\log{p}\bigg| \ll \exp(-c_4
\frac{\log{N}}{\log{P}}) $$ for $\exp(\log^{\frac{1}{2}}{N}) \leq P
\leq N^{c_5}$ and if the exceptional character occurs, the right
hand side  may be replaced by $ (1-\tilde{\beta})\log{P} \exp(-c_4
\frac{\log{N}}{\log{P}}).$ Here the term with $q=1$ is read as
follows: if there is no exceptional character, it is
$$
\sum_{\substack{x-h \\ p \equiv i \pmod{g}}}^{x}\log{p} -
\sum_{\substack{x-h < n \leq x \\ n>0 \\ n \equiv i \pmod{g}}}1
$$
and if there is the exceptional character, it is
$$
\sum_{\substack{x-h \\ p \equiv i
\pmod{g}}}^{x}\tilde{\chi}(p)\log{p} + \sum_{\substack{x-h < n \leq
x \\ n>0 \\ n \equiv i \pmod{g}}}n^{\tilde{\beta}-1}.
$$
\end{lem}

For a Dirichlet character $\chi$ modulo $q$, define
$$
S_{i}(\chi,\eta) := \sum_{\substack{P < p \leq X \\ p \equiv i
\pmod{g}}}(\log{p})\chi(p)e(p\eta),
$$
and
$$
T_{i}(\eta) := \sum_{\substack{P<n\leq X \\ n \equiv i
\pmod{g}}}e(n\eta), \qquad \widetilde{T}_{i}(\eta) :=
-\sum_{\substack{P<n\leq X \\ n \equiv i
\pmod{g}}}n^{\tilde{\beta}-1}e(n\eta),
$$
where the last one is defined only if there is an exceptional
character.

\noindent Let $\chi_0$ be the principal character modulo $q$. Define
$$
W_{i}(\chi, \eta) := \left\{ \begin{array}{lll} S_{i}(\chi, \eta) -
T_{i}(\eta) &
\textrm{ if } \chi = \chi_{0}, \\
S_{i}(\chi, \eta) - \widetilde{T}_{i}(\eta) &
\textrm{ if } \chi = \tilde{\chi}\chi_{0} ,\\
S_{i}(\chi, \eta) & \textrm{ otherwise. }
\end{array} \right.
$$

\noindent Suppose that a Dirichlet character $\chi \pmod{q}$ is
induced by a primitive character $\chi^{*} \pmod{r}$. Put
$$
W^{A}_{i}(\chi) = \big( \int^{\frac{1}{rQ}}_{-\frac{1}{rQ}}
\left|W_{i}(\chi, A\eta)\right|^2d\eta \big)^{\frac{1}{2}} \qquad
\textrm{and} \qquad W^{A}_{i} = \sum_{q \leq
P}\sum_{\chi}\!{}^*W^{A}_{i}(\chi).
$$
We note that $W^{A}_{i}(\chi)=W^{A}_{i}(\chi^{*})$. Then we have the
following lemma which is a modification of \cite[(7.1)]{MV}.

\begin{lem} \label{bddW}
If the exceptional character does not occur, there is an absolute
constant $c_6$ which satisfies
$$
W_{i}^{A} \ll X^{\frac{1}{2}}\exp(-c_6 \frac{\log{X}}{\log{P}}),
$$
and if the exceptional character occurs, the right hand side may be
replaced by $X^{\frac{1}{2}}(1-\tilde{\beta})\log{P}\exp(-c_6
\frac{\log{X}}{\log{P}})$.
\end{lem}

\begin{proof}
First assume that $\chi$ be a primitive character which is not equal to
 $\chi_0$ nor $ \tilde{\chi}\chi_0$. Then
$$
W_{i}^{A}(\chi)^2 = \int^{\frac{1}{qQ}}_{-\frac{1}{qQ}}
\bigg|\sum_{\substack{P<p\leq X \\ p \equiv i \pmod{g}}}
\chi(p)\log{p}\cdot e(Ap\eta)\bigg|^2d\eta.
$$

\noindent Applying \cite[Lemma 4.2]{MV} to the real numbers
$$
u_{n} := \left\{ \begin{array}{ll}
\chi(p)\log{p} & \textrm{if $n = Ap, \,\,P<p \leq X, \,\,p \equiv i \pmod{g}$,} \\
0 & \textrm{otherwise,}
\end{array}\right.
$$
we get
$$
W_{i}^{A}(\chi)^2 \ll \int^{2AX}_{0}
\bigg|\frac{1}{qQ} \sum_{\substack{P<p\leq X
\\ x-\frac{qQ}{2}\leq Ap \leq x \\ p \equiv i \pmod{g}}}
\chi(p)\log{p}\bigg|^2 dx.
$$

\noindent Thus
\begin{eqnarray*}
W_{i}^{A}(\chi) &\ll& (2AX)^{\frac{1}{2}}\max_{x \leq
2AX}\frac{1}{qQ}
\bigg|\sum_{\substack{x-\frac{qQ}{2} \leq Ap \leq x \\
p \equiv i \pmod{g}}}\chi(p)\log{p}\bigg|\\
&\ll& X^{\frac{1}{2}}\max_{x \leq 2AX}(Q+\frac{qQ}{2})^{-1}
\bigg|\sum_{\substack{x-\frac{qQ}{2} \leq Ap \leq x \\
p \equiv i \pmod{g}}}\chi(p)\log{p}\bigg|\\
&\leq & X^{\frac{1}{2}} \max_{x \leq 2AX}\max_{h \leq X}(Q+h)^{-1}
\bigg|\sum_{\substack{x-h \leq Ap \leq x \\ p \equiv i
\pmod{g}}}\chi(p)\log{p}\bigg|
\,\,\,\,\,\,\,\,\,\,\,\,\,\,\,\,(\because \frac{qQ}{2}
\leq X)\\
&\leq& X^{\frac{1}{2}}\max_{x \leq 2X}\max_{h \leq 2X}(2Q+h)^{-1}
\bigg|\sum_{\substack{x-h \leq p \leq x \\ p \equiv i
\pmod{g}}}\chi(p)\log{p}\bigg|.
\end{eqnarray*}

\noindent When $\chi = \chi_0$ or $\tilde{\chi}\chi_0$,
$W_{i}(\chi,\eta)$ is exactly the term of the case of $q=1$ in Lemma
\ref{Majorlem2}. Hence
\begin{eqnarray}
W_{i}^{A} & = & \sum_{q \leq P}\sum_{\chi}\!{}^{*}W_{i}^{A}(\chi) \nonumber \\
& \ll & X^{\frac{1}{2}}\sum_{q \leq P}\sum_{\chi}\!{}^{*} \max_{x
\leq
2X}\max_{h \leq 2X}(h+\frac{2X}{P})^{-1} \bigg|\sum_{\substack{x-h \\
p \equiv i \pmod{g}}}^{x} \chi(p)\log{p} \bigg| . \nonumber
\end{eqnarray}

\noindent So the lemma is proved by Lemma \ref{Majorlem2}.
\end{proof}

Now we can prove the following analogue of  \cite[Lemma 2]{BKW}. For
the proof, we follow the proof of  \cite[Lemma 2]{BKW}. New
ingredients in our proof are the estimations of the bounds of new
terms which do not appear in the proof of  \cite[Lemma 2]{BKW},
using Lemma 9 if needed.

\begin{prop} \label{Majormain}
Suppose that $Y$ is a real number  with $1 \leq Y \leq
X^{\frac{\delta} {k}}$. If there is at least one integer $m$ such
that
$$ 2f(m) \equiv Ai + Bj \pmod{g},$$
then
$$ r_{\Gamma}(2f(n); \mathfrak{M}) \gg XY^{-\frac{1}{2}}(\log{X})^{-1} $$
for all $n \in (\kappa N, N]$ with $ 2f(n) \equiv Ai + Bj \pmod{g}$
except at most $O(N^{1+\epsilon}Y^{-1})$ numbers.
\end{prop}

\begin{proof}
Let $n \in (\kappa N, N]$ be an integer with $ 2f(n) \equiv Ai + Bj
\pmod{g}$. The proof is divided into three steps. In the first step,
we will prove that if there is no exceptional character, then
$r_{\Gamma}(2f(n);\mathfrak{M}) \gg X.$ Next we will show that even
if there is the exceptional character, the same lower bound holds
when $(2f(n), \tilde{r}) = 1$. Finally we will show that the number
of integers $n$ for which
$$
(2f(n), \tilde{r}) >1 \,\,\, \mbox{and} \,\,\,
r_{\Gamma}(2f(n);\mathfrak{M}) \ll XY^{-\frac{1}{2}}(\log{X})^{-1}
$$
is at most $O(N^{1+\epsilon}Y^{-1})$.

First we assume that there is no exceptional character. For $\alpha
\in \mathfrak{M}(q,a)$ we write $\alpha = \frac{a}{q} + \eta$ for
$(a,q)=1$, $0 \leq \eta <1$ and $q<P$. Let $\chi$ be a Dirichlet
character of modulus $q$ and $\tau(\chi) =
\sum_{n=1}^{q}\chi(n)e(\frac{n}{q})$ the Gaussian sum. Since
$\displaystyle e(\frac{a}{q}) =
\frac{1}{\varphi(q)}\sum_{\chi}\bar{\chi}(a)\tau(\chi)$, we have
$$
S_{i}(A\alpha) = \frac{\mu(q)}{\varphi(q)}T_{i}(A\eta) +
\frac{1}{\varphi(q)}\sum_{\chi}\chi(Aa)\tau(\bar{\chi})W_{i}(\chi,
A\eta).
$$

\noindent Thus
\begin{eqnarray}
&&r_{\Gamma}(2f(n);\mathfrak{M})\nonumber\\
&=& \sum_{q \leq P}  \frac{\mu(q)^2}{\varphi(q)^{2}}c_{q}(-2f(n))
\int^{\frac{1}{qQ}}_{-\frac{1}{qQ}}T_{i}(A\eta)T_{j}(B\eta)e(-2f(n)\eta)d\eta \label{firstterm}
    \\
&+& \sum_{q \leq P} \frac{\mu(q)}{\varphi(q)^2}\sum_{\chi^{\prime}}
\chi^{\prime}(B)c_{\chi^{\prime}}(-2f(n))\tau(\bar{\chi'})
\int^{\frac{1}{qQ}}_{-\frac{1}{qQ}}T_{i}(A\eta)W_{j} (\chi^{\prime},
B\eta)e(-2f(n)\eta)d\eta \label{secondterm}
     \\
&+& \sum_{q \leq P}
\frac{\mu(q)}{\varphi(q)^2}\sum_{\chi}\chi(A)c_{\chi}(-2f(n))
\tau(\bar{\chi})\int^{\frac{1}{qQ}}_{-\frac{1}{qQ}}T_{j}
(B\eta)W_{i}(\chi, A\eta)e(-2f(n)\eta)d\eta \label{thirdterm}
   \\
&+& \sum_{q \leq P} \bigg[ \frac{1}{\varphi(q)^2}\sum_{\chi,
\chi^{\prime}}\chi(A)\chi^{\prime}(B)
c_{\chi\cdot\chi^{\prime}}(-2f(n))\tau(\bar{\chi})\tau(\bar{\chi}^{\prime}) \nonumber \\
&&\qquad \qquad \qquad \qquad \qquad \times
\int^{\frac{1}{qQ}}_{-\frac{1}{qQ}}W_{i} (\chi,
A\eta)W_{j}(\chi^{\prime}, B\eta)e(-2f(n)\eta)d\eta \bigg], \label{fourthterm}
\end{eqnarray}

\noindent where $c_q(m)= \sum_{(a,q)=1}^{q}e(\frac{am}{q})$ and
$c_{\chi}(m) = \sum_{h=1}^{q}\chi(h)e(\frac{hm}{q})$. Using
\cite[Lemma 5.5]{MV} and the same argument in \cite[Section 6]{MV},
we have
$$
(\ref{secondterm}) \ll \frac{2f(n)}{\varphi(2f(n))}W_i^AX^{\frac{1}{2}}, \hskip 1em
(\ref{thirdterm}) \ll  \frac{2f(n)}{\varphi(2f(n))}W_j^BX^{\frac{1}{2}} \hskip 1em
\mbox{and} \hskip 1em
(\ref{fourthterm}) \ll
\frac{2f(n)}{\varphi(2f(n))}W_i^AW_j^B.
$$

\noindent Let us compute the bound of (\ref{firstterm}). Since $T_{i}(A\eta) \ll
\frac{1}{\|gA\eta \|}$, by assuming the harmless condition $qQ
> 2gA$ and $A \geq B$, we have
\begin{eqnarray*}
&&\int^{\frac{1}{2gA}}_{\frac{1}{qQ}}T_{i}(A\eta)T_{j}(B\eta)e(-2f(n)\eta)d\eta \ll
\int^{\frac{1}{2gA}}_{\frac{1}{qQ}} \frac{1}{\|gA\eta \|} \frac{1}{\|gB\eta
\|}d\eta\leq \int^{\frac{1}{2gA}}_{\frac{1}{qQ}} \frac{1}{g^2AB\eta^2}d\eta \\&\ll&
qQ.
\end{eqnarray*}

\noindent By some elementary computations, we have
\begin{eqnarray*}
&&\int^{\frac{1}{2g}}_{\frac{1}{2gA}}T_{i}(A\eta)T_{j}(B\eta)e(-2f(n)\eta)d\eta
-\mathop{\sum\sum}_{\substack{P < k,l \leq X \\ k \equiv i, l \equiv j \\Ak + Bl = 2f(n)}}
(\frac{1}{2g}-\frac{1}{2gA})\\
&=&\mathop{\sum\sum}_{\substack{P < k,l \leq X \\ k \equiv i, l
\equiv j}} \frac{1}{(Ak+Bl-2f(n))2\pi i}[e(\frac{Ak+Bl-2f(n)}{2g})
-e(\frac{Ak+Bl-2f(n)}{2gA})]\\
&=& \sum_{1\leq t \leq gA}(\mathop{\sum\sum}_{\substack{P <k,l \leq X \\
Ak + Bl - 2f(n) \equiv t
\\ \pmod{2gA}
\\ k \equiv i, l \equiv j \pmod{g}}}\frac{1}{Ak+Bl-2f(n)}
-\mathop{\sum\sum}_{\substack{P <k,l \leq X \\ Ak + Bl - 2f(n) \equiv gA
+ t \\ \pmod{2gA}
\\ k \equiv i, l \equiv j \pmod{g}}}\frac{1}{Ak+Bl-2f(n)}) \\
& & \qquad \qquad \qquad
\times \frac{1}{2\pi i}(e(\frac{t}{2g}) - e(\frac{t}{2gA}))\\
&=& O(\log{X}),
\end{eqnarray*}

\noindent and
\begin{eqnarray*}
&&\int^{\frac{1}{2g}}_{-\frac{1}{2g}}T_{i}(A\eta)T_{j}(B\eta)e(-2f(n)\eta)d\eta
=\mathop{\sum\sum}_{\substack{P < k,l \leq X \\ k \equiv i, l \equiv
j}} \int^{\frac{1}{g}}_{0}e((Ak+Bl-2f(n))\eta)d\eta \\
&=&
\mathop{\sum\sum}_{\substack{P < k,l \leq X \\ k \equiv i, l \equiv
j \\ Ak+Bl = 2f(n)}} \frac{1}{g} = \frac{2f(n)}{g^2AB} + O(1).
\end{eqnarray*}

\noindent Together with these three estimations, we have
$$
\int^{\frac{1}{qQ}}_{-\frac{1}{qQ}}T_{i}(A\eta)T_{j}(B\eta)e(-2f(n)\eta)d\eta
= \frac{2f(n)}{g^2AB} + O(qQ).
$$

\noindent Thus
$$
(\ref{firstterm})=\sum_{q \leq P} \frac{\mu(q)^2}{\varphi(q)^{2}}c_{q}(-2f(n))
(\frac{2f(n)}{g^2AB} + O(qQ)).
$$

\noindent By \cite[6.13 and 6.14]{MV}, we conclude that
$$
r_{\Gamma}(2f(n);\mathfrak{M}) =
\mathfrak{S}(2f(n))\frac{2f(n)}{g^2AB} + O(X^{1+\delta}P^{-1}) +
O(\frac{2f(n)}{\varphi(2f(n))}(W_i^AX^{\frac{1}{2}}+W_j^BX^{\frac{1}{2}}+W_i^AW_j^B)),
$$
where
$$
\displaystyle \mathfrak{S}(n) = \sum_{q=1}^{\infty}
\frac{\mu(q)^2}{\varphi(q)^{2}}c_{q}(-n) = \prod_{p \nmid
n}(1-\frac{1}{(p-1)^2})\prod_{p |n}(1+\frac{1}{p-1}).
$$

\noindent In this equation, the first error term
$X^{1+\delta}P^{-1}$ is negligible. Also from Lemma \ref{bddW}, the
second error term is less then
$\frac{6f(n)}{\varphi(2f(n))}Xe^{-\frac{c_6}{6\delta}}$. If we
choose a sufficiently small positive real number  $\delta$, then
$r_{\Gamma}(2f(n);\mathfrak{M}) \geq (1-c_7)
\mathfrak{S}(2f(n))f(n)$. These imply that
$r_{\Gamma}(2f(n);\mathfrak{M}) \gg X$, which is the conclusion of
the first step.

Next, assume that there is the exceptional character. In this case,
we have
$$
S_{i}(A\alpha) = \frac{\mu(q)}{\varphi(q)}T_{i}(A\eta) + \frac{1}{\varphi(q)}
\sum_{\chi}[\chi(Aa)\tau(\bar{\chi})W_{i}(\chi, A\eta)]
 + \frac{\tilde{\chi}(Aa)\tau(\tilde{\chi}\chi_{0})}{\varphi(q)}
 \widetilde{T}_{i}(A\eta).
$$

\noindent Note that the last term appears when $q$ divides
$\tilde{r}$, the modulus of the exceptional character.
This makes
additional terms in $r_{\Gamma}(2f(n); \mathfrak{M})$ which are
\begin{eqnarray}
&&\sum_{\substack{q \leq P \\ \tilde{r}|q}}
\frac{\tau(\tilde{\chi}\chi_{0})^2}{\varphi(q)^2} \tilde{\chi}(AB)
c_{q}(-2f(n))\int^{\frac{1}{qQ}}_{-\frac{1}{qQ}}
\widetilde{T}_{i}(A\eta) \widetilde{T}_{j}(B\eta) e(-2f(n)\eta)
d\eta
    \nonumber \\
&+&  \sum_{\substack{q \leq P \\ \tilde{r}|q}}
\frac{\mu(q)\tau(\tilde{\chi}\chi_{0})}{\varphi(q)^2}
\tilde{\chi}(B) c_{\tilde{\chi}\chi_{0}}(-2f(n))
\int^{\frac{1}{qQ}}_{-\frac{1}{qQ}}T_{i}(A\eta)\widetilde{T}_{j}
(B\eta)e(-2f(n)\eta)d\eta
    \nonumber \\
&+& \sum_{\substack{q \leq P \\ \tilde{r}|q}}
\frac{\mu(q)\tau(\tilde{\chi}\chi_{0})}{\varphi(q)^2}
\tilde{\chi}(A) c_{\tilde{\chi}\chi_{0}}(-2f(n))
\int^{\frac{1}{qQ}}_{-\frac{1}{qQ}}T_{j}(B\eta)\widetilde{T}_{i}
(A\eta)e(-2f(n)\eta)d\eta
    \nonumber\\
&+& \sum_{\substack{q \leq P \\ \tilde{r}|q}} \frac{\tilde{\chi}(B)
\tau(\tilde{\chi}\chi_{0})} {\varphi(q)^2}
\sum_{\chi}c_{\tilde{\chi}\cdot\chi}(-2f(n))\tau(\bar{\chi}) \chi(A)
\int^{\frac{1}{qQ}}_{-\frac{1}{qQ}}W_{i}(\chi,
A\eta)\widetilde{T}_{j} (B\eta)e(-2f(n)\eta)d\eta
    \nonumber \\
&+&\sum_{\substack{q \leq P \\ \tilde{r}|q}}
\frac{\tilde{\chi}(A) \tau(\tilde{\chi}\chi_{0})} {\varphi(q)^2}
\sum_{\chi}c_{\tilde{\chi}\cdot\chi}(-2f(n))\tau(\bar{\chi}) \chi(B)
\int^{\frac{1}{qQ}}_{-\frac{1}{qQ}}W_{j}(\chi,
B\eta)\widetilde{T}_{i} (A\eta)e(-2f(n)\eta)d\eta.
\qquad \label{exceptchar1}
\end{eqnarray}

\noindent By the same argument in the previous step, the last two
terms of (\ref{exceptchar1}) are bounded by
$$
2X^{\frac{1}{2}}\sum_{\substack{q \leq P \\ \tilde{r}|q}}
\frac{1}{\varphi(q)^2} \sum_{\chi}
|c_{\tilde{\chi}\chi}(-2f(n))\tau(\bar{\chi})
\tau(\tilde{\chi}\chi_{0})|W_{k}^{C}(\chi) \ll
\frac{2f(n)}{\varphi(2f(n))}X^{\frac{1}{2}}W_{k}^{C},
$$

\noindent where $C = A$ or $B$ and $ k = i$ or $j$. And the first
three terms in (\ref{exceptchar1}) turn out to be
\begin{eqnarray}
&&\sum_{\substack{q \leq P \\ \tilde{r}|q}}
\frac{\tau(\tilde{\chi}\chi_{0})^2}{\varphi(q)^2}
\tilde{\chi}(ABa^2) c_{q}(-2f(n)) (\widetilde{I}^{AB}_{ij}(2f(n)) +
O(qQ))
    \nonumber \\
&+& \sum_{\substack{q \leq P \\ \tilde{r}|q}}
\frac{\mu(q)\tau(\tilde{\chi}\chi_{0})}{\varphi(q)^2}
\tilde{\chi}(B) c_{\tilde{\chi}\chi_{0}}(-2f(n))
(\widetilde{J}_{ji}^{BA}(2f(n)) + O(qQ))
    \nonumber \\
&+& \sum_{\substack{q \leq P \\ \tilde{r}|q}}
\frac{\mu(q)\tau(\tilde{\chi}\chi_{0})}{\varphi(q)^2}
\tilde{\chi}(A) c_{\tilde{\chi}\chi_{0}}(-2f(n)) (
\widetilde{J}_{ij}^{AB}(2f(n)) + O(qQ)), \nonumber
\label{exceptchar2}
\end{eqnarray}

\noindent where
$$
\widetilde{I}^{AB}_{ij}(n) := \int^{\frac{1}{g}}_{0}
\widetilde{T}_{i}(A\eta) \widetilde{T}_{j}(B\eta) e(-n\eta)d\eta
\,\,\,\mbox{and}\,\,\, \widetilde{J}_{ij}^{AB}(n) := \int^{\frac{1}{g}}_{0}
\widetilde{T}_{i}(A\eta) T_{j}(B\eta) e(-n\eta)d\eta.
$$

\noindent Let $\displaystyle
\widetilde{\mathfrak{S}}(n) := \sum_{\substack{q=1 \\
\tilde{r}|q}}^{\infty}
 \frac{\tau(\tilde{\chi}\chi_{0})^2}{\varphi(q)^{2}}c_{q}(-n)$. Then
we can prove
\begin{eqnarray}
r_{\Gamma}(2f(n);\mathfrak{M}) &=&
\mathfrak{S}(2f(n))\frac{2f(n)}{g^2AB} +
\tilde{\chi}(ABa^2)\widetilde{\mathfrak{S}}(2f(n))\widetilde{I}^{AB}_{ij}(2f(n))
    \nonumber \\
&+& O(\frac{\tilde{\chi}(2f(n))^2 \tilde{r} \cdot
2f(n)X}{\varphi(\tilde{r})^2\varphi(2f(n))}) +
O(X^{1+\delta}P^{-1}(2f(n),\tilde{r}))
\nonumber \\
&+& O(\frac{2f(n)} {\varphi({2f(n)})}(X^{\frac{1}{2}}(W_{i}^{A} +
W_{j}^{B}) + W_{i}^{A}W_{j}^{B})). \label{con1}
\end{eqnarray}
by just following \cite[p.364]{MV}. Our assumption
$(2f(n),\tilde{r})=1$ means that the fourth term of (\ref{con1}) is
less then $X^{1-5\delta}$. Using the same method in \cite[section 8]{MV}, we
have
$$
\tilde{\mathfrak{S}}(n) \ll o(1) \,\,\,\,\mbox{and}\,\,\,\,
\widetilde{I}^{AB}_{ij}(n) = \sum_{\substack{P < k < n-P \\ k \equiv
i, \frac{n-Ak}{B} \equiv j
 \\ \frac{n-Ak}{B} \in \mathbb{Z}}}(k(\frac{n-Ak}{B}))^{\tilde{\beta}-1}
\leq \sum_{\substack{P < k <n-P \\\frac{n-Ak}{B}
\in \mathbb{Z}}}(k(\frac{n-Ak}{B}))^{\tilde{\beta}-1} \leq n^{\tilde{\beta}}.
$$

\noindent These facts and Lemma \ref{bddW} imply that
$$
r_{\Gamma}(2f(n); \mathfrak{M}) \gg X
\qquad \textrm{if $(2f(n), \tilde{r}) =1$,}
$$
which is the conclusion of the second step.

Finally we assume $(2f(n),\tilde{r})>1$, so we have
$\tilde{\chi}(2f(n)) = 0$. Then by Lemma \ref{bddW} and (\ref{con1}) 
there is a constant $c_7$ satisfying 
\begin{equation}
|r_{\Gamma}(2f(n);\mathfrak{M}) - \mathfrak{S}(2f(n))
\frac{2f(n)}{g^2AB} - \widetilde{\mathfrak{S}}(2f(n))
\widetilde{I}_{ij}^{AB}(-2f(n))| \leq c_{7} (T_{1}+T_{2}),
\label{con2}
\end{equation}

\noindent where
$$
T_1 = X^{1+\delta}P^{-1}(2f(n),\tilde{r}) \qquad \mbox{and} \qquad
T_2 = \frac{2f(n)}{\varphi(2f(n))}(1-\tilde{\beta})
Xe^{\frac{-c_6}{\delta}}\log{P}.
$$

\noindent All the arguments in \cite[p.122--123]{BKW} can be used
for $r_{\Gamma}(2f(n); \mathfrak{M})$ since we know that
$\widetilde{I}^{AB}_{ij}(2f(n)) \leq n^{\tilde{\beta}}$. As a
consequence, we conclude that if $1 < (2f(n),\tilde{r}) < Y$, then
$$
r_{\Gamma}(2f(n); \mathfrak{M}) \gg XY^{-1}(\log{X})^{-1}
$$

\noindent and there are at most $O(N^{1+\epsilon}Y^{-1})-$exceptions
which are $n$ with $(2f(n),\tilde{r}) > Y$. This completes the proof
of the proposition.
\end{proof}

\section{Proof of Theorem 2}
We note that if there is at least one integer $m$ such that $ 2f(m)
\equiv Ai + Bj \pmod{g},$ the set of $n \in (\kappa N, N]$ with $
2f(n) \equiv Ai + Bj \pmod{g}$ has a positive density in the set of
$n \in (\kappa N, N]$. Then the proof of Theorem 2 is exactly same
as the proof of \cite[Theorem 1]{BKW} since Proposition
\ref{Minormain} and Proposition \ref{Majormain} give the same
results for $r_{\Gamma}(2f(n); \mathfrak{m})$ and $r_{\Gamma}(2f(n);
\mathfrak{M})$ as \cite[Lemma 1]{BKW} and \cite[Lemma
2]{BKW}.\bigskip

\noindent {\bf Acknowledgment.} The authors thank Professor Trevor
Wooley for useful discussions and comments.

\bigskip

\noindent Department of Mathematics, Seoul National University,
Seoul, Korea

\noindent E-mail: dhbyeon@snu.ac.kr

\noindent Department of Mathematics, Seoul National University,
Seoul, Korea

\noindent E-mail: kyjeongg@gmail.com


\begin{thebibliography}{99}

\bibitem[BJ]{BJ} D. Byeon and K. Jeong,
\emph{Sums of two rational cubes with many prime factors}, preprint.

\bibitem[BJK]{BJK} D. Byeon, D. Jeon and C. Kim,
\emph{Rank one quadratic twists of an infinite family of elliptic
curves}, J. Reine und Angew. Math., 633 (2009), 67--76.

\bibitem[BKW]{BKW} J. Br\"udern, K. Kawada and T. D. Wooley,
\emph{Additive representation in thin sequences, II: The binary
Goldbach problem}, Mathematika \noindent{\bf 47}, (2000), 117--125.

\bibitem[Dav]{Dav} H. Davenport,
\emph{Multiplicative number theory}, Springer-Verlag, Second
Edition, (1980).

\bibitem[Gal]{Ga} P. X. Gallagher,
\emph{A large sieve density estimate near $\sigma = 1$}, Invent.
Math. \noindent{\bf 11}, (1970), 329--339.

\bibitem[IK]{IK} H. Iwaniec and E. Kowalski,
 \emph{Analytic number theory}, American Mathematical Society, (2004).

\bibitem[MV]{MV} H. L. Montgomery and R. C. Vaughan,
\emph{The exceptional set in Goldbach`s problem}, Acta Arith.
\noindent{\bf 27.1} (1975), 353--370.



\end{thebibliography}
\end{document}